\documentclass[11pt]{article}
\usepackage{amssymb, amsmath, amsthm, graphicx, hyperref}
\usepackage[left=1in,top=1in,right=1in]{geometry}
\date{}
\allowdisplaybreaks

\theoremstyle{definition}
\newtheorem{theorem}{Theorem}[section]
\newtheorem{lemma}[theorem]{Lemma}

\newtheorem{proposition}[theorem]{Proposition}
\newtheorem{remark}[theorem]{Remark}
\newtheorem{claim}[theorem]{Claim}

\newtheorem{problem}[theorem]{Problem}
\counterwithin{equation}{section}

\title{Unbalanced Zarankiewicz problem for bipartite subdivisions \\ with applications to incidence geometry}
\author{
Lili Ködmön\thanks{ELTE Eötvös Loránd University, Budapest, Hungary. Supported by the NRDI Fund EXCELLENCE\_24-151504 project ``Combinatorics and Geometry''. Email:{\tt kodmonlili@student.elte.hu}}
\and
Anqi Li\thanks{Department of Mathematics, Stanford University. Supported in part by Jane Street. Email:{\tt aqli@stanford.edu}}
\and
Ji Zeng\thanks{Alfréd Rényi Institute of Mathematics, Budapest, Hungary. Supported by ERC Advanced Grants ``GeoScape'', no. 882971 and ``ERMiD'', no. 101054936. Also partly supported by NSF grant DMS-1928930 while in residence at the Simons--Laufer Mathematical Sciences Institute during the Spring 2025 semester. Email:{\tt jzeng@ucsd.edu}.}
}

\begin{document}

\maketitle

\begin{abstract}
For a bipartite graph $H$, its linear threshold is the smallest real number $\sigma$ such that every bipartite graph $G = (U \sqcup V, E)$ with unbalanced parts $|V| \gtrsim |U|^\sigma$ and without a copy of $H$ must have a linear number of edges $|E| \lesssim |V|$. We prove that the linear threshold of the complete bipartite subdivision graph $K_{s,t}'$ is at most $\sigma_s = 2 - 1/s$. Moreover, we show that any $\sigma < \sigma_s$ is less than the linear threshold of $K_{s,t}'$ for sufficiently large $t$ (depending on $s$ and $\sigma$).

Some geometric applications of this result are given: we show that any $n$ points and $n$ lines in the complex plane without an $s$-by-$s$ grid determine $O(n^{4/3 - c})$ incidences for some constant $c > 0$ depending on $s$; and for certain pairs $(p,q)$, we establish nontrivial lower bounds on the number of distinct distances determined by $n$ points in the plane under the condition that every $p$ points determine at least $q$ distinct distances.
\end{abstract}

\section{Introduction}

The \textit{Zarankiewicz problem} asks for the maximum number of edges, denoted by $z(m,n;s,t)$, an $m$-by-$n$ bipartite graph $G$ can have without containing a copy of $K_{s,t}$, i.e. an $s$-by-$t$ complete bipartite graph. Here, ``containing a copy'' requires slightly more than ``containing a subgraph'': the $s$-side of $K_{s,t}$ must sit in the $m$-side of $G$. The most fundamental result towards the Zarankiewicz problem is the upper bound of K\H{o}v\'ari, S\'os, and Tur\'an~\cite{kHovari1954problem}:\begin{equation*}
    z(m,n;s,t) \leq O(mn^{1-1/s} + n).
\end{equation*} From this bound, we notice that $n \gtrsim m^{s}$ implies $z(m,n;s,t) \lesssim n$. Motivated by this observation, we make the following definition: the \textit{linear threshold} of a bipartite graph $H$ is the smallest real number $\sigma$ satisfying for any $\epsilon>0$ there exists $\delta$ such that every bipartite graph $G = (U \sqcup V, E)$ with $|V| \geq \epsilon|U|^\sigma$ and $|E| \geq \delta|V|$ must contain a copy of $H$. For example, our observation from the K\H{o}v\'ari--S\'os--Tur\'an bound says that the linear threshold of $K_{s,t}$ is at most $s$.

It is natural to consider the linear thresholds of other forbidden subgraphs. Erd\H{o}s and F\"uredi \cite{furedi1991turan} showed that the linear threshold of a bipartite graph $G = (U \sqcup V, E)$ is at most $\Delta(V)$, that is, the \textit{maximum degree} of the right hand side vertex set $V$. (We remark that only the $\Delta(V) = 2$ case was stated in \cite{furedi1991turan}.) The \textit{complete bipartite subdivision} $K_{s,t}'$ is the bipartite graph obtained by substituting each edge of $K_{s,t}$ by a path of length two. Here, we put the subdivision vertices on the right hand side of $K_{s,t}'$ as a convention. (More formally, we define $K_{s,t}'$ by vertex set $\{u_i,v_j\} \sqcup \{w_{ij}\}$ and edge set $\{u_iw_{ij}, v_jw_{ij}\}$ for all $1\leq i\leq s$, $1\leq j\leq t$.) Hence their result~\cite{furedi1991turan} says the linear threshold of $K_{s,t}'$ is at most 2. In this paper, we give an improved bound that is not far from the true answer.

\begin{theorem}\label{main}
    For every integer $s \geq 1$, the linear threshold of $K_{s,t}'$ is at most $\sigma_s = 2 - 1/s$. Moreover, any $\sigma < \sigma_s$ is less than the linear threshold of $K_{s,t}'$ for sufficiently large $t$ (depending on $s$ and $\sigma$).
\end{theorem}

The \textit{forbidden subgraph problem} asks for the maximum number of edges, denoted by $\text{ex}(n, G)$, a graph can have without containing a subgraph $G$. At an asymptotic level, we can relate $\text{ex}(n,K_{s,t})$ with $z(n,n;s,t)$. Hence this problem can be considered a variation of the Zarankiewicz problem in a balanced setting. There is a series of studies on $\text{ex}(n,K_{s,t}')$ (see e.g. \cite{janzer2019improved,sudakov2020turan,conlon2021more,conlon2021extremal,kang2021rational}), and they demonstrate the phenomenon that $\text{ex}(n, K_{s,t}')$ is smaller than $\Theta(n^{3/2})$, the general bound for bipartite graphs with right hand side maximum degree two (\cite{furedi1991turan,alon2003turan}). However, it is worth noting that $\text{ex}(n, C_6) = O(n^{4/3})$ where $C_6$ is the $6$-cycle graph, but a construction by de Caen and Sz\'ekely~\cite{de1991maximum} shows that 2 is the linear threshold of $C_6$. Thus, a smaller $\text{ex}(n, G)$ does not guarantee a smaller linear threshold of $G$. On the other hand, the linear threshold of $C_{2k}$ for $k \geq 4$ is always smaller than 2 by a result of Naor and Verstra\"ete~\cite{naor2005note}, despite $\text{ex}(n, C_k) = (\frac{1}{4} + o(1))n^2$ whenever $k \geq 3$ is odd. Thus, the linear threshold of $G'$ (subdivision of $G$) is not determined by the asymptotic behavior of $\text{ex}(n,G)$.

\medskip

Many results in incidence geometry follow divide-and-conquer strategies with space-partitioning techniques (see e.g. \cite{matouvsek1991efficient,kaplan2012simple}). Moreover, it is a common practice to upper bound incidences by linear terms in unbalanced parts of the partitioning, and our notion of linear thresholds measures this level of unbalance. As a specific example, we point out the following point-variety incidence bound which motivated our study. For a point set $P$ and a collection of algebraic varieties $\mathcal{V}$, both in $\mathbb{R}^d$, we define their incidence graph $I(P,\mathcal{V})$ as the bipartite graph with vertex set $P \sqcup \mathcal{V}$ and edge set $\{(p, V) \in P \times \mathcal{V}: p\in V\}$. We denote the size of its edge set by $|I(P,\mathcal{V})|$.

\begin{theorem}\label{threshold2incidence}
    Let $H$ be a bipartite graph with its linear threshold at most $\sigma$. If $P$ is a set of $m$ points and $\mathcal{V}$ is a set of $n$ constant-degree algebraic varieties, both in $\mathbb{R}^d$, such that $I(P,\mathcal{V})$ does not contain a copy of $H$, then for every $\epsilon > 0$ and as $m,n\to \infty$, we have\begin{equation*}
        |I(P,\mathcal{V})| \leq O\left( m^{\frac{(d-1)\sigma}{d\sigma-1}+\epsilon}n^{\frac{d(\sigma-1)}{d\sigma-1}} + m + n \right).
    \end{equation*}
\end{theorem}

A point-line configuration $(P, L)$ is called an \textit{$s$-by-$s$ grid} if the line collection $L$ is the union of $L_1$ and $L_2$ each containing $s$ lines, and the point collection $P$ equals $\{\ell_1 \cap \ell_2 : \ell_1 \in L_1, \ell_2 \in L_2 \}$ which has size exactly $s^2$ (square of $s$). Mirzaei and Suk \cite{mirzaei2021grids} proved that point-line configurations in the real plane without a grid determine $O(n^{4/3 - c})$ incidences for some $c > 0$. This is better than the celebrated Szemer\'edi--Trotter bound $O(n^{4/3})$, which has corresponding extremal constructions but with grid substructures (\cite{szemeredi1983extremal,elekes2001sums}). While the argument in \cite{mirzaei2021grids} is based on space-partitioning, Balko and Frankl~\cite{balko2024forbidden} recently gave a different proof using crossing-number-related techniques (\cite{szekely1997crossing,pach1999new}). It is unclear how to generalize these proofs to the complex plane where the analogous Szemer\'edi--Trotter-type theorems are much more involved (\cite{toth2015szemeredi,zahl2015szemeredi}). Nevertheless, we employ a technique of Solymosi and Tao \cite{solymosi2012incidence}, together with our Theorems~\ref{main} and~\ref{threshold2incidence}, to prove the following bound.

\begin{theorem}\label{grid}
    For fixed integer $s \geq 1$ and real number $\epsilon > 0$, any $n$ points $P$ and $n$ lines $L$ in the complex plane $\mathbb{C}^2$ without an $s$-by-$s$ grid sub-configuration satisfy \begin{equation*}
        |I(P,L)| \leq O\left(n^{\frac{4}{3} - \frac{1}{9s-6} + \epsilon}\right).
    \end{equation*}
\end{theorem}

A classical problem by Erd\H{o}s is to study the extremal function $f(n,p,q)$, the number of colors required to edge-color $K_n$ such that every $p$ edges receive at least $q$ distinct colors. Using the probabilistic method, Erd\H{o}s and Gy\'arf\'as~\cite{erdHos1997variant} showed that \begin{equation*}
    f(n,p,q) \leq c n^{\frac{p-2}{\binom{p}{2}-q+1}} ~\text{where $c$ depends only on $p,q$.}
\end{equation*} A geometric variation of $f(n,p,q)$ is $d(n,p,q)$, the smallest number of distinct distances determined by $n$ points in the plane under the condition that every $p$ points determine at least $q$ distinct distances. Due to its geometric nature, one would expect $d(n,p,q)$ to behave quite differently from $f(n,p,q)$. One such geometric result that we are aware of is due to Fox, Pach, and Suk~\cite{fox2018more}:\begin{equation*}
    d\left( n,p,\binom{p}{2}-p+6 \right) \geq n^{\frac{8}{7} - o(1)}.
\end{equation*}

Using our Theorems~\ref{main} and~\ref{threshold2incidence}, we derive lower bounds for $d(n,p,q)$ that will contrast the Erd\H{o}s--Gy\'arf\'as bound~\cite{erdHos1997variant} for certain new pairs of $(p,q)$. Our bounds somewhat interpolate the result of Fox--Pach--Suk~\cite{fox2018more} and another lower bound by Pohoata and Sheffer~\cite{PS19} for $f(n,p,q)$.
\begin{theorem}\label{distance}
    For fixed integer $s\geq 1$ and real number $\epsilon > 0$, we have\begin{equation*}
        d\left(n, p, \binom{p}{2} - p + 3 \cdot \left\lfloor \frac{p}{2s} \right\rfloor + 2s + 2 \right) \geq \Omega\left( n^{\frac{8}{7} + \frac{18}{7(7s-4)} -\epsilon} \right).
    \end{equation*}
\end{theorem}

\medskip

The rest of this paper is organized as follows. Sections~\ref{sec_main1} is devoted to the proof of Theorem~\ref{main}. In Section~\ref{sec_threshold2incidence}, we outline a proof of Theorem~\ref{threshold2incidence} for completeness. Sections~\ref{sec_grid} and~\ref{sec_distance} contain the proofs of Theorems~\ref{grid} and~\ref{distance} respectively. We conclude with some remarks in Section~\ref{sec_remark}. We systematically omit floors and ceilings whenever they are not crucial for the sake of clarity.

\section{Linear thresholds of bipartite subdivisions}\label{sec_main1}

We shall prove the first half of Theorem~\ref{main} in this section. The following lemma originates from a result of Erd\H{o}s and Simonovits (Theorem~1 in \cite{erdos1970some}, see also \cite{jiang2012turan,conlon2021more}). Here, our treatment is more complicated due to the host graph being sparse.

\begin{lemma}\label{biregular}
    For each integer $s \geq 2$, there is a sufficiently small positive constant $c_s$ such that any bipartite graph $G=(U \sqcup V,E)$ with $|V| \geq |U|^{2-1/s}$ and $|E|\geq \delta|V|$ ($1 \leq \delta \in \mathbb{R}$) contains a subgraph $\tilde{G} = (\tilde{U} \sqcup \tilde{V}, \tilde{E})$ with $|\tilde{U}| \geq |U|^{c_s}$ satisfying the following conditions: \begin{align*}
        &\makebox[0pt][l]{$\text{(I)}~|\tilde{V}| \geq |\tilde{U}|^{2-1/s},$}\phantom{\text{(III)}~c_s \Delta_{\tilde{G}}(\tilde{V}) \leq |\tilde{E}|/|\tilde{V}|,}
        \qquad\qquad \text{(II)}~|\tilde{E}|\geq c_s \delta|\tilde{V}|,\\
        &\text{(III)}~c_s \Delta_{\tilde{G}}(\tilde{V}) \leq |\tilde{E}|/|\tilde{V}|,
        \qquad\qquad \text{(IV)}~c_s \delta^{c_s} |\tilde{U}|\Delta_{\tilde{G}}(\tilde{U})^{1- 1/s} \leq |\tilde{E}|.
    \end{align*}
\end{lemma}

\begin{proof}
    We construct $\tilde{G}$ by two iterative procedures. In the first procedure, we let $G_0 = (U_0\sqcup V_0, E_0) = G$ and construct $G_i = (U_i\sqcup V_i, E_i)$ as follows. Given $G_i$ already constructed, we set $V_{i+1}$ to be the $|V_i|/16$ many vertices in $V_i$ with largest degrees. If less than half of $E_i$ are adjacent to $V_{i+1}$, we terminate the first procedure. Otherwise, by a simple averaging argument, we can find $U_{i+1} \subset U_i$ with $|U_{i+1}| = |U_i|/16^{\frac{s}{2s-1}}$ such that \begin{equation*}
        |E_{G_i}(U_{i+1}, V_{i+1})| \geq \frac{1}{2}|E_i|/16^{\frac{s}{2s-1}}.
    \end{equation*} Here, $E_{G_i}(U_{i+1}, V_{i+1})$ denotes the set of edges in $G_i$ between $U_{i+1}$ and $V_{i+1}$. Then we take $G_{i+1}$ as the subgraph of $G_i$ induced on $U_{i+1} \sqcup V_{i+1}$. Let $\ell$ be the index where the first procedure terminates at $G_{\ell}$. We have the following identities:\begin{equation*}
        |U_{\ell}| = |U_0|/16^{\frac{\ell s}{2s-1}},\quad |V_\ell| = |V_0|/16^\ell,\quad |E_\ell| \geq \frac{1}{2^\ell}|E_0|/16^{\frac{\ell s}{2s-1}}.
    \end{equation*} Combining them with $|U_\ell||V_\ell|\geq|E_\ell|$ and $|E_0| \geq |V_0|$, we can solve for \begin{equation}\label{eq:biregular_ell}
        \ell \leq \log(|U_0|)/3.
    \end{equation}

    Next, we let $\tilde{V} = V_\ell \setminus V_{\ell+1}$. Notice that $|\tilde{V}| = (15/16)|V_\ell|$ and $|E_{G_\ell}(U_\ell, \tilde{V})| \geq |E_\ell|/2$. By an averaging argument, we can find $U' \subset U_\ell$ with $|U'| = (15/16)^{\frac{s}{2s-1}}|U_\ell|$ such that the subgraph $G'\subset G_\ell$ induced on $U' \sqcup \tilde{V}$ satisfies $|E'| \geq (15/16)^{\frac{s}{2s-1}}|E_\ell|/2$, where $E'$ is the edge set of $G'$. Since $V_{\ell+1}$ is chosen among vertices with largest degrees, \begin{equation}\label{eq:biregular_B}
        \Delta_{G'}(\tilde{V}) \leq \frac{|E_{\ell}|}{|V_{\ell+1}|} \leq 30 \cdot (16/15)^{\frac{s}{2s-1}} \cdot \frac{|E'|}{|\tilde{V}|}.
    \end{equation}
    
    In the second procedure, we let $G'_0 = (U'_0 \sqcup \tilde{V}, E'_0) =G'$ and construct $G'_i = (U'_i \sqcup \tilde{V}, E'_i)$ as follows. Given $G'_i$ already constructed, we set $U'_{i+1}$ to be the $|U'_i|^{1-1/s}$ many vertices in $U'_i$ with largest degrees. If less than half of $E'_i$ are adjacent to $U'_{i+1}$, we terminate the second procedure. In this case, we let $\tilde{U} = U'_i \setminus U'_{i+1}$ and $\tilde{G}$ be the subgraph of $G'_i$ induced on $\tilde{U} \sqcup \tilde{V}$ and denote its edge set as $\tilde{E}$. Since $U'_{i+1}$ is chosen among vertices with largest degrees, we have \begin{equation}\label{eq:biregular_A1}
        \Delta_{\tilde{G}}(\tilde{U}) \leq \frac{|E'_i|}{|U'_{i+1}|} \leq \frac{2|\tilde{E}|}{|U'_{i+1}|} = \frac{2|\tilde{E}|}{|U'_i|^{1-1/s}} \leq \frac{2|\tilde{E}|}{|\tilde{U}|^{1-1/s}}.
    \end{equation} If more than half of $E'_i$ are adjacent to $U'_{i+1}$, we take $G'_{i+1}$ as the subgraph of $G'_i$ induced on $U'_{i+1}\sqcup \tilde{V}$. Suppose the second procedure has not terminated after $s\log(s)$ iterations, then we terminate it and define $\tilde{U} = U'_{s\log(s)}$ and $\tilde{G} = G'_{s\log(s)}$. In this case, we have \begin{equation}\label{eq:biregular_A2}
        |\tilde{U}| = |U'_0|^{(1-1/s)^{s\log(s)}} \leq |U'_0|^{1/s} \leq |\tilde{V}|^{1/s}.
    \end{equation}

    Now we check that $\tilde{G}$ satisfies the conditions we claimed. Using \eqref{eq:biregular_ell} from the first procedure, we can check that $|U_\ell| \geq |U_0|^{\frac{2s-3}{6s-3}}$. It is a consequence of the second procedure terminating in at most $s \log s $ iterations that $|\tilde{U}| \geq |U'_0|^{(1-1/s)^{s\log(s)}}$. Hence $|\tilde{U}| \geq |U|^{c_s}$ is satisfied for small enough $c_s$.

    It is easy to see that $|V_i| \geq |U_i|^{2-1/s}$ are guaranteed throughout the first procedure, and $|\tilde{V}| \geq |U'|^{2-1/s} \geq |\tilde{U}|^{2-1/s}$. So condition (I) is satisfied. We can inductively check that $|E_i| \geq \delta|V_i|$ in the first procedure. It is clear from the second procedure that $|\tilde{E}|$ and $|E_\ell|$ only differ by a multiplicative constant depending on $s$. This verifies condition (II) for small enough $c_s$. Similarly, we can verify condition (III) using \eqref{eq:biregular_B}.

    Finally, if the second procedure terminated with \eqref{eq:biregular_A1}, we can compute \begin{align*}
        |\tilde{U}|\Delta_{\tilde{G}}(\tilde{U})^{1 - 1/s} &\leq |\tilde{U}|\left( \frac{2|\tilde{E}|}{|\tilde{U}|^{1-1/s}} \right)^{1 - 1/s} = \left( |\tilde{U}|^{2-1/s} \right)^{1/s}\left( 2|\tilde{E}| \right)^{1 - 1/s} \leq |\tilde{V}|^{1/s}\left( 2|\tilde{E}| \right)^{1 - 1/s}\\
        &\leq \left( c_s^{-1}\delta^{-1}|\tilde{E}| \right)^{1/s}\left( 2|\tilde{E}| \right)^{1 - 1/s} = 2 \cdot \left( 2c_s\delta \right)^{-1/s}|\tilde{E}|.
    \end{align*} If the second procedure terminated with \eqref{eq:biregular_A2}, we have \begin{equation*}
         |\tilde{U}|\Delta_{\tilde{G}}(\tilde{U})^{1 - 1/s} \leq |\tilde{V}|^{1/s}|\tilde{V}|^{1-1/s} \leq |\tilde{V}| \leq c_s^{-1}\delta^{-1}|\tilde{E}|.
    \end{equation*} Hence we can verify condition (IV) for small enough $c_s$. Overall, the proof is finished.
\end{proof}

The following proposition is the main part of Theorem~\ref{main}. Its proof follows the argument of Theorem~4.1 in the paper~\cite{conlon2021more} by Conlon, Janzer, and Lee.
\begin{proposition}\label{cjl_Kst}
    For integers $t \geq s > 1$, there exists a sufficiently large constant $\delta$ such that the following holds. If a bipartite graph $G = (U \sqcup V, E)$ satisfies the conditions for $\tilde{G} = (\tilde{U} \sqcup \tilde{V}, \tilde{E})$ in Lemma~\ref{biregular} with parameters $s$, $c_s$, and $\delta$, then $G$ contains a copy of $K_{s,t}'$.
\end{proposition}

\begin{proof}

    For a contradiction, suppose that $G$ does not contain a copy of $K_{s,t}'$. Write $|U| = m$, $|V| = n$, and $|E| = e$. We consider a weighted graph $W$ with vertex set being $U$ and each pair $\{u,v\} \subset U$ is assigned with weight \begin{equation*}
        W(u,v) = |N_G(\{u,v\})|.
    \end{equation*} Here, the \textit{neighborhood} $N_G(\{u,v\})$ is the set of vertices adjacent to all of $\{u,v\}$ in $G$. We say that $\{u,v\}$ is a \textit{light edge} if $1\leq W(u,v) < \binom{s+t}{2}$, and is a \textit{heavy edge} if $W(u,v) \geq \binom{s+t}{2}$.
    \begin{claim}\label{cjl_43}
        If $W(U) := \sum_{\{u,v\} \in \binom{U}{2}} W(u,v) \geq 8(s+t+1)^2n$, then the number of light edges in $W$ is at least $\frac{W(U)}{4(s+t+1)^3}$.
    \end{claim}
    \noindent This claim follows from Lemma~10 in \cite{janzer2019improved} (see also Lemma~4.3 in \cite{conlon2021more}) so we omit its proof. Roughly speaking, if we have an abundant number of heavy edges, then we can embed the complete subdivision graph $K_{s+t}'$ (which contains $K_{s,t}'$) into $G$ by a greedy process. Consequently, most weights should contribute to light edges, hence creating many of them.
    
    Using double counting and Jensen's inequality, we can compute $W(U) \geq n\binom{e/n}{2}\geq \frac{e^2}{4n}$. Using condition (II) from Lemma~\ref{biregular}, we can guarantee the hypothesis of Claim~\ref{cjl_43} by taking $\delta$ to be sufficiently large relative to $s,t,c_s$.

    For distinct $u_1,\dots,u_s \in U$, we write $N'_W(u_1,\dots,u_s)$ for the set of $x\in U \setminus \{u_1,\dots,u_s\}$ for which there exist distinct $v_1,\dots,v_s \in V$ such that $v_i\in N_{G}(\{u_i,x\})$. We claim the following estimate: \begin{equation}\label{eq:cjl_Kst}
        \sum_{\tau\in \binom{U}{s}}|N'_W(\tau)| \geq \Omega\left(\frac{e^{2s}}{m^{s-1}n^{s}}\right).
    \end{equation} We prove this by counting the number of $(s+1)$-tuples $(x,u_1,\dots,u_s) \in U^{s+1}$ such that \begin{itemize}
        \item[(i)] Each $xu_i$ is a light edge in $W$.
        \item[(ii)] For any $i \neq j$, we have $N_G(x)\cap N_G(u_i)\cap N_G(u_j) = \emptyset$.
    \end{itemize} For any $x\in U$, let $d_1(x)$ be the number of light edges adjacent to $x$ in $W$. Then by Claim~\ref{cjl_43}, we have $\sum_x d_1(x) \geq \Omega(e^2/n)$. The number of $(s+1)$-tuples satisfying (i) is \begin{equation*}
        \sum_x d_1^s(x) \geq m \left(\frac{\sum_x d_1(x)}{m}\right)^s \geq \Omega\left(\frac{e^{2s}}{m^{s-1}n^{s}}\right).
    \end{equation*} On the other hand, we can enumerate all the $(s+1)$-tuples not satisfying (ii). First, we choose a pair $i,j$ that violates (ii) which has $s^2$ options. Second, we choose $x$ from $U$ which has $m$ options. Then, we choose $v\in N_G(x)$ that is adjacent to the potential $u_i,u_j$, and we have at most $\Delta(U)$ options. Next, we choose $u_i,u_j$ from $N_G(v)$ which have at most $\Delta(V)^2$ options. Finally, we choose the remaining $u_i$'s as neighbors of neighbors of $x$, and this gives at most $(\Delta(U)\Delta(V))^{s-2}$ options. Hence the number of $(s+1)$-tuples satisfying (i) but not (ii) is at most \begin{equation*}
        s^2 \cdot m \cdot \Delta(U)^{s-1} \cdot \Delta(V)^{s} \leq s^2 \cdot \left( c_s^{-s} \delta^{-sc_s}\frac{e^s}{m^{s-1}} \right) \cdot \left( c_s^{-1} \frac{e^s}{n^s} \right) \leq \sum_x d_1^s(x).
    \end{equation*}
    Here, the first inequality can be checked using conditions (III) and (IV) from Lemma~\ref{biregular}, and the second inequality follows from choosing a sufficiently large $\delta$. Notice that any $(s+1)$-tuple satisfying (i) and (ii) implies $x\in N'_W(u_1,\dots,u_s)$, so we can conclude \eqref{eq:cjl_Kst}.

    We also need the following fact about $K_{s,t}'$-copy-free graphs. We omit the proof as it is almost identical to Corollary~4.6 in \cite{conlon2021more}.
    \begin{claim}\label{cjl46}
        Let $G$ be a $K_{s,t}'$-copy-free bipartite graph with vertex bipartition $U \sqcup V$. Suppose that for some distinct $u_1,\dots,u_s\in U$, $|N_W'(u_1,\dots,u_s)| \geq C$ for some sufficiently large $C$ depending on $s,t$. Then there exist $v\in V$, $1\leq k\leq s$, and $\Omega(|N_W'(u_1,\dots,u_s)|^s)$ many $s$-sets $\{x_1,\dots,x_s\}\subset N_W'(u_1,\dots,u_s)$ such that $\{x_1,\dots,x_s\}\cup \{u_k\}\subset N_G(v)$ and no $x_ix_j$ is a heavy edge.
    \end{claim}

    Using \eqref{eq:cjl_Kst}, $e\geq c_s \delta n$, and $n\geq m^{2-1/s}$ (from Lemma~\ref{biregular}), we can guarantee that\begin{equation*}
        \sum_{\tau\in \binom{U}{s}}|N'_W(\tau)| \geq \Omega\left(\frac{e^{2s}}{m^{s-1}n^{s}}\right) \geq 2Cm^s,
    \end{equation*} by taking $\delta$ sufficiently large relative to $C$, where $C$ is the constant required in Claim~\ref{cjl46}. As a consequence of this inequality, writing $U_s'$ as the collection of $\tau \in \binom{U}{s}$ with $|N'_W(\tau)| \geq C$, we have\begin{equation*}
        \sum_{\tau\in U_s'} |N'_W(\tau)| \geq \Omega\left(\frac{e^{2s}}{m^{s-1}n^{s}}\right).
    \end{equation*} By Claim~\ref{cjl46}, for each $\tau = \{u_1, \dots, u_s\}\in U_s'$, there exist $v\in V$, some $1\leq k\leq s$, and $\Omega(|N'_W(\tau)|^s)$ many $s$-sets $\{x_1,\dots,x_s\}\subset N'_W(\tau)$ such that $\{x_1,\dots,x_s\}\cup \{u_k\}\subset N_G(v)$ and no $x_ix_j$ is a heavy edge. We use this information to estimate the number of $2s$-tuples $(x_1,\dots,x_s,u_1,\dots,u_s) \in U^{2s}$ with the following properties:\begin{itemize}
        \item[(i)] All $x_i$ and $u_j$ are distinct.
        \item[(ii)] There exist $v\in V$ and $1\leq k \leq s$ such that $\{x_1,\dots,x_s\}\cup\{u_k\}\subset N_G(v)$.
        \item[(iii)] For each $i,j$, $N_G(x_i)\cap N_G(u_j)\neq \emptyset$.
        \item[(iv)] No $x_ix_j$ determines a heavy edge in $W$.
    \end{itemize}
    By Jensen's inequality and the lower bound above, the number of such $2s$-tuples is at least\begin{equation*}
        |U'_s| \cdot \Omega\left(\left(\frac{\sum_{\tau\in U_s'} |N'_W(\tau)|}{|U'_s|}\right)^s\right) \geq \Omega\left( m^s\left(\frac{\frac{e^{2s}}{m^{s-1}n^{s}}}{m^s}\right)^s \right) \geq \omega\left(n\Delta(V)^{2s}\right).
    \end{equation*} Here, the last inequality is checked by conditions (I)--(III) from Lemma~\ref{biregular}, and the $\omega$-notation is understood as $m \to \infty$. On the other hand, there are at most $ns\Delta(V)^{s+1}$ ways to choose $v,k,x_1,\dots,x_s,u_k$ such that property (ii) holds. Thus for at least one such choice, there are at least $\omega(\Delta(V)^{s-1})$ ways to extend it to a suitable $2s$-tuple. Therefore, we conclude that there exist distinct $x_1,\dots,x_s \in U$ such that no $x_ix_j$ is heavy and the number of $u \in U$ with $N_G(x_i)\cap N_G(u) \neq \emptyset$ for all $i$ is at least $\omega(\Delta(V))$. 
    
    Finally, we argue the existence of $K_{s,t}'$ based on these $x_1,\dots,x_s$. Since no $x_ix_j$ is heavy, we have $|\bigcup_{i<j} (N_G(x_i)\cap N_G(x_j))| = O(1)$, so the number of $u \in U$ with $N_G(x_i)\cap N_G(x_j)\cap N_G(u)\neq \emptyset$ for some $1\leq i\leq j\leq s$ is $O(\Delta(V))$. Hence there is a set $A \subset U$ of $\omega(\Delta(V))$ many vertices distinct from $x_1,\dots,x_s$, such that for each $u \in A$ there are distinct $v_1,\dots,v_s \in V$ with $v_i\in N_G(x_i)\cap N_G(u)$. Take a maximal subset $ A'= \{u_1,\dots,u_r\}\subset A$ such that there exist distinct $w_{ij}\in N_G(x_i)\cap N_G(u_j)$ for $1\leq i\leq s$ and $1\leq j\leq r$. If $r\geq t$, we have found a copy of $K_{s,t}'$ inside $G$. So we assume that $r<t$. For any $u \in A \setminus A'$, there exists distinct $v_i \in N_G(x_i) \cap N_G(v)$. By the maximality of $A'$, we must have $v_i = w_{jk}$ for some $i,j,k$. So, we have $u \in N_G(w_{jk})$ for this $j,k$. Therefore, \begin{equation*}
        |A| \leq |A'| + |\bigcup_{1\leq i\leq s,~1\leq j\leq r} N_G(w_{ij})| \leq t + st\Delta(V),
    \end{equation*} which contradicts $|A| \geq \omega(\Delta(V))$. This concludes the proof.
\end{proof}

\begin{proof}[Proof of Theorem~\ref{main}]
    Without loss of generality, we can assume $s \leq t$. Our main goal is to show that for any $\epsilon>0$ there exists $\delta$ such that every bipartite graph $G = (U \sqcup V, E)$ with $|V| \geq \epsilon|U|^{\sigma_s}$ and $|E| \geq \delta|V|$ must contain a copy of $K_{s,t}'$.

    The $s=1$ case is simple: we can delete every vertex $v \in V$ with degree $d(v) < t$ while the resulting graph still contains half of its original edges (due to a large enough $\delta$); then we can find a vertex $u \in U$ with $d(u) \geq t$ and greedily embed a copy of $K_{1,t}'$ using $N_G(a)$. It is easy to argue that any number smaller than $\sigma_1 = 1$ cannot be a linear threshold of $K_{1,2}'$. For example, we can add isolated vertices to the right hand side of a star graph $K_{m,1}$.

    To deal with the $s \geq 2$ case, we observe that by an averaging argument, a bipartite graph $G = (U \sqcup V, E)$ with $|V| \geq \epsilon|U|^\sigma$ and $|E| \geq \delta |V|$ contains a subgraph $G' = ( U'\sqcup V', E')$ with $|V'| \geq |U'|^{\sigma}$ and $|E'| \geq \epsilon^{\frac{1}{\sigma}}\delta|V'|$. Hence, we can apply Lemma~\ref{biregular} and Proposition~\ref{cjl_Kst} to $G'$.

    The ``Moreover'' part follows from a simple probabilistic construction: include each edge in an $m$-by-$n$ bipartite random graph with probability $p$ independently, and delete a vertex on the left hand side from each $K_{s,t}'$ in the random graph. We can upper bound the expected number of $K_{s,t}'$ in this random graph by $O(m^{s+t}n^{st}p^{2st})$, and we notice the expected number of edges is $mnp$. Hence by setting \begin{equation*}
        m^{s+t}n^{st}p^{2st} \cdot n \ll mnp \quad\text{and}\quad n \ll mnp,
    \end{equation*} we can guarantee the resulting graph contains $\omega(n)$ many edges and is $K_{s,t}'$-copy-free. These condition can be satisfied with $n \ll m^{2 - 1/s -1/t}$. Hence it suffices to take $1/t < \sigma_s - \sigma$ to conclude $\sigma$ is smaller than the linear threshold of $K_{s,t}'$.
\end{proof}

\section{Incidence bounds via linear thresholds}\label{sec_threshold2incidence}

We outline a proof of Theorem~\ref{threshold2incidence} in this section for the sake of completeness. This theorem is essentially due to Fox, Pach, Sheffer, Suk, and Zahl~\cite{fox2017semi} where the special case $H = K_{s,t}$ and $\sigma = s$ is proven as their Theorem~1.2. In particular, all our ``variety''-related notions are identical to those in \cite{fox2017semi}, which are standard in the study of real algebraic geometry. The proof is an induction on the following more general statement.

\begin{proposition}
Let $H$ be a bipartite graph with its linear threshold at most $\sigma$. Suppose $P$ is a set of $m$ points and $\mathcal{V}$ is a set of $n$ constant-degree algebraic varieties, both in $\mathbb{R}^d$, such that the incidence graph $I(P,\mathcal{V})$ does not contain a copy of $H$. Suppose furthermore that $P$ is fully contained in an irreducible variety $V$ of dimension $e$ and degree $D$, and no member $S\in \mathcal{V}$ contains $V$. Then for every $\epsilon > 0$ and as $m,n\to \infty$, we have
\begin{equation}\label{threshold2incidence1}
|I(P,\mathcal{V})| \leq \alpha_{1,e}m^{\frac{(e-1)\sigma}{e\sigma-1}+\epsilon}n^{\frac{e(\sigma-1)}{e\sigma-1}}+\alpha_{2,e}(m+n),
\end{equation}
where $\alpha_{1,e}$ and $\alpha_{2,e}$ are constants depending on $H,\sigma,d,e,D, \epsilon$ and the degrees of varieties in $\mathcal{V}$.
\end{proposition}

\begin{proof}

We will induct both on $e$ and $m+n$. Clearly, one can take $V=\mathbb{R}^d$ in the above statement to conclude Theorem~\ref{threshold2incidence}. First, we consider the base case for the induction. If $m+n$ is small, then \eqref{threshold2incidence1} is implied by choosing sufficiently large $\alpha_{1,e}$ and $\alpha_{2,e}$. Similarly, when $e=0$, \eqref{threshold2incidence1} is also guaranteed by choosing sufficiently large $\alpha_{1,0}$ and $\alpha_{2,0}$.

By the definition of linear threshold, we can assume without loss of generality that $n < m^\sigma$. Otherwise, there exists a constant $\delta$ depending on $H$ such that $|I(P,\mathcal{V})| < \delta n$ and we can take $\alpha_{2,e} > \delta$ to ensure \eqref{threshold2incidence1}. In particular, we have\begin{equation}\label{threshold2incidence2}
    n = n^\frac{e-1}{e\sigma-1}n^\frac{e(\sigma-1)}{e\sigma-1} < m^\frac{(e-1)\sigma}{e\sigma-1}n^\frac{e(\sigma-1)}{e\sigma-1}.
\end{equation}

Next, we consider the inductive step. Assume that $\eqref{threshold2incidence1}$ holds when $|P|+|\mathcal{V}| < m + n$ or $\dim V < e$. By polynomial partitioning (Theorem~4.2 in \cite{fox2017semi}, see also~\cite{guth2015erdHos}), we find a polynomial $f$ of degree at most $C_{part} \cdot r^{1/e}$, where $C_{part}$ depends only on $e$ and $r$ shall be determined later, such that $f$ does not completely vanish on $V$ and no connected components of $V \setminus Z(f)$ contains more than $(1/r)$-fraction of $P$.

Let $\Omega_1,\Omega_2,\dots,\Omega_z$ be the connected components of $V\setminus Z(f)$. By a Milnor--Thom-type result (e.g. Theorem~A.2 in \cite{solymosi2012incidence}), there exists a constant $C_{cell}$ depending on $d,e,D$ such that $z \leq C_{cell} \cdot r$. We divide the edge set of $I(P,\mathcal{V})$ into three subsets:
\begin{itemize}
\item $I_1$ consists of the incidences $(p,S)$ such that $p\in Z(f)$ and $S$ properly intersects every irreducible component of $V \cap Z(f)$ that contains $p$.
\item $I_2$ consists of the incidences $(p,S)$ such that $p$ is contained in an irreducible component of $V \cap Z(f)$ that is fully contained by $S$.
\item $I_3$ consists of incidences not in $I_1$ or $I_2$, that is, all of $(p,S)$ such that $p$ is not contained in $V \cap Z(f)$.
\end{itemize}

\noindent \textbf{Bounding $|I_1|$}: Let $V' = V \cap Z(f)$ and let $m_0 = |P \cap V'|$. Since $f$ does not vanish on $V$, $V'$ is a variety of dimension $e' \leq e-1$ and can be written as a union of $\gamma_1$ irreducible (over $\mathbb{R}$) components, each of dimension at most $e'$ and degree at most $\gamma_2$. Here $\gamma_1$ and $\gamma_2$ only depends $D$, $e$, $d$ and the degree of $f$ (see e.g. \cite{galligo1995complexity}). We can apply the inductive hypothesis to obtain\begin{equation*}
    |I_1| \leq  \gamma_1 \cdot \alpha_{1,e-1}m_0^{\frac{(e-2)\sigma}{(e-1)\sigma-1}+\epsilon}n^{\frac{(e-1)(\sigma-1)}{(e-1)\sigma-1}}+\alpha_{2,e-1}(m_0+n).
\end{equation*} Here, the constants $\alpha_{1,e-1}$ and $\alpha_{2,e-1}$ are independent of $m$ and $n$. Using $n<m^\sigma$, we can check that\begin{equation*}
    m^{\frac{(e-2)\sigma}{(e-1)\sigma-1}}n^{\frac{(e-1)(\sigma-1)}{(e-1)\sigma-1}} \leq  m^{\frac{(e-1)\sigma}{e\sigma-1}}n^{\frac{e(\sigma-1)}{e\sigma-1}}.
\end{equation*} (A detailed computation can be found as equation (12) in \cite{fox2017semi}.) Combining above inequalities, absorbing the $a_{2,e-1}\cdot n$ term by \eqref{threshold2incidence2}, and choosing sufficiently large $\alpha_{1,e},\alpha_{2,e}$, we can conclude\begin{equation}\label{boundingI1}
    |I_1| \leq \frac{\alpha_{1,e}}{3}m^{\frac{(e-1)\sigma}{e\sigma-1}+\epsilon}n^{\frac{e(\sigma-1)}{e\sigma-1}}+\frac{\alpha_{2,e}}{2}m_0.
\end{equation}

\noindent \textbf{Bounding $|I_2|$}: Denote the vertex set of $H$ as $A \sqcup B$. Note that $V'$ can be written as a union of $\gamma_1$ irreducible (over $\mathbb{R}$) components, each of dimension at most $e-1$, where $\gamma_1$ depends only on $D$, $e$, $d$ and the degree of $f$. Since the incidence graph of $I(P,\mathcal{V})$ does not contain a copy of $H$, each of the irreducible components of $V'$ either contains less than $|A|$ points of $P$, or is contained in less than $|B|$ members of $\mathcal{V}$. The contribution to $|I_2|$ from the first scenario is at most $|A| \gamma_1 n$, and from the second scenario is at most $|B|m_0$. Combining these bounds with \eqref{threshold2incidence2} and choosing $\alpha_{1,e},\alpha_{2,e}$ to be sufficiently large depending on $H$ and $\gamma_1$, we have
\begin{equation}\label{boundingI2}
|I_2| \le \frac{\alpha_{1,e}}{3}m^{\frac{(e-1)s}{es-1}+\epsilon}n^{\frac{e(s-1)}{es-1}}+\frac{\alpha_{2,e}}{2}m_0.
\end{equation}

\noindent \textbf{Bounding $|I_3|$}: Let $P_i := P\cap \Omega_i$ and $\mathcal{V}_i$ be the set of elements in $\mathcal{V}$ that properly intersect $\Omega_i$. Write $m_i = |P_i|$ and $n_i = |\mathcal{V}_i|$. By the property of the polynomial $f$, we have an upper bound $m_i \leq m/r$. Note also that $m' := \sum_{i=1}^z m_i = m-m_0$. By a B\'ezout-type theorem, see e.g. \cite{solymosi2012incidence} Theorem~A.2, and proper intersection, there exists a constant $C_{inter}$ such that \begin{equation*}
    \sum_{i=1}^z n_i \leq C_{inter} \cdot \kappa n r^{(e-1)/e},
\end{equation*} where $\kappa$ is the maximum degree of the varieties in $\mathcal{V}$. By H\"older's inequality, we have\begin{align*}
    \sum_{i=1}^z n_i^{\frac{e(\sigma-1)}{e\sigma-1}} &\leq \left( \sum_{i=1}^z n_i \right)^{\frac{e(\sigma-1)}{e\sigma-1}} \left( \sum_{i=1}^z 1 \right)^{\frac{e-1}{e\sigma-1}}\\
    &\leq \left( C_{inter}\cdot\kappa n r^{(e-1)/e} \right)^{\frac{e(\sigma-1)}{e\sigma-1}} \left( C_{cell} \cdot r \right)^{\frac{e-1}{e\sigma-1}}\\
    &\leq C_{Hold} \cdot n^{\frac{e(\sigma-1)}{e\sigma-1}} r^{\frac{(e-1)\sigma}{e\sigma-1}}, 
\end{align*} where $C_{Hold}$ depends on $\kappa,C_{inter},C_{cell},e,\sigma$.

By the induction hypothesis, we have \begin{align}
|I_3| = \sum_{i=1}^z |I(P_i,\mathcal{V}_i)| &\le \sum_{i=1}^z \left(\alpha_{1,e} m_i^{\frac{(e-1)\sigma}{e\sigma-1}+\epsilon}n_i^{\frac{e(\sigma-1)}{e\sigma-1}}+\alpha_{2,e}(m_i+n_i)\right) \notag\\
&\le \alpha_{1,e} m^{\frac{(e-1)\sigma}{e\sigma-1}+\epsilon}\left(r^{\frac{(e-1)\sigma}{e\sigma-1}+\epsilon}\right)^{-1} \sum_{i=1}^z n_i^{\frac{e(\sigma-1)}{e\sigma-1}} + \sum_{i=1}^z\alpha_{2,e}(m_i+n_i) \notag\\
&= \alpha_{1,e} C_{Hold}\cdot r^{-\epsilon}m^{\frac{(e-1)\sigma}{e\sigma-1}+\epsilon}n^{\frac{e(\sigma-1)}{e\sigma-1}} + \alpha_{2,e}m'+ \alpha_{2,e}n C_{inter}\cdot \kappa r^{(e-1)/e} \notag\\
&\leq \frac{\alpha_{1,e}}{3}m^{\frac{(e-1)\sigma}{e\sigma-1}+\epsilon}n^{\frac{e(\sigma-1)}{e\sigma-1}} + \alpha_{2,e}m'.\label{boundingI3}
\end{align} Here in the last step, we first choose $r$ to be sufficiently large depending on $\epsilon$ and $C_{Hold}$ such that $C_{Hold}\cdot r^{-\epsilon}\leq 1/6$. Then we use \eqref{threshold2incidence2} and choose $\alpha_{1,e}$ to be sufficiently large depending on $r, C_{inter},\kappa,\alpha_{2,e}$, absorbing the last term into the first term.

\medskip

We conclude the inductive process by combining \eqref{boundingI1}, \eqref{boundingI2}, and \eqref{boundingI3}.
\end{proof}

\section{Grid-free complex point-line arrangements}\label{sec_grid}

In this section, we prove the theorem below from which Theorem~\ref{grid} follows as a corollary. Our proof is inspired by a constant-degree-polynomial-partitioning technique of Solymosi and Tao~\cite{solymosi2012incidence}. Here, a \textit{2-flat} refers to a 2-dimensional affine subspace.

\begin{theorem}\label{grid_2flat}
Let $P$ be a set of $m$ points and $\mathcal{F}$ be a set of $n$ 2-flats, both in $\mathbb{R}^4$, such that every pair of 2-flats in $\mathcal{F}$ has at most one intersection, and the incidence graph $I(P,\mathcal{F})$ does not contain a copy of $K_{s,t}'$. Then for every $\epsilon > 0$, we have
\begin{equation}\label{eq:grid_2flat}
    |I(P,\mathcal{F})| \leq C_1 m^{\frac{2s-1}{3s-2}+\epsilon}n^{\frac{2s-2}{3s-2}}+C_2(m+n),
\end{equation}
where $C_1, C_2$ are large constants depending on $s,t,\epsilon$.
\end{theorem}

\begin{proof}

We shall prove \eqref{eq:grid_2flat} by induction on $m$, where $C_1$ and $C_2$ are to be chosen later. The claim holds trivially for $m=0$. We assume that $m \geq 1$ and \eqref{eq:grid_2flat} holds for smaller values in place of $m$.

Observe that we may restrict our attention to values of $n$ such that \begin{equation*}
    m^{\frac{1}{2}}\leq n \leq m^{2-\frac{1}{s}}.
\end{equation*} After choosing $C_2$ to be large enough (depending on $s,t$), the first inequality comes from the K\H{o}v\'ari--S\'os--Tur\'an bound together with the hypothesis that two flats intersect at most once (see Section~3 in \cite{solymosi2012incidence}), and the second inequality comes from our Theorem~\ref{main}.

We may assume that every point has at least half of the average incident 2-flats. If a point is incident to fewer 2-flats, we delete it. By polynomial partitioning (Corollary~5.5 in \cite{solymosi2012incidence}, see also~\cite{guth2015erdHos}), there exists a degree $D < (12M)^{1/4}$ polynomial $Q$ which partitions $\mathbb{R}^4 \setminus \{ Q=0 \}$ into $M$ components
\begin{equation*}
    \mathbb{R}^4 = \{ Q=0 \} \cup \Omega_1 \cup \dots \cup \Omega_M
\end{equation*} such that no component contains more than $m/M$ points, where M is a constant depending on $\epsilon$ and $s$ to be determined later.

By the Harnack curve theorem~\cite{harnack1876ueber}, each 2-flat not on the surface $\{Q=0\}$ meets at most $D^2$ cells in $\{\Omega_i\}$. Let $\mathcal{F}_i$ denote the set of 2-flats in $\mathcal{F}$ intersecting $\Omega_i$, then we have
\begin{equation*}
    \sum_{i=1}^{M} |\mathcal{F}_i| \leq D^2n.
\end{equation*}

Now using the induction hypothesis, we have
\begin{align*}
    |I(P \cap \Omega_i,\mathcal{F}_i)| 
    &\leq C_1 |P \cap \Omega_i|^{\frac{2s-1}{3s-2}+\epsilon} |\mathcal{F}_i|^{\frac{2s-2}{3s-2}} + C_2 \left( \frac{m}{M} + |\mathcal{F}_i| \right)\\
    &\leq C_1 \left( \frac{m}{M} \right) ^{\frac{2s-1}{3s-2}+\epsilon} |\mathcal{F}_i|^{\frac{2s-2}{3s-2}} + C_2 \left( \frac{m}{M} + |\mathcal{F}_i| \right).
\end{align*}

Also, by H\"older's inequality, we have
\begin{equation*}
    \sum_{i=1}^{M} |\mathcal{F}_i|^{\frac{2s-2}{3s-2}} \leq \left( \sum_{i} |\mathcal{F}_i| \right)^{\frac{2s-2}{3s-2}} M^{\frac{s}{3s-2}} \leq \left( D^2 n \right)^{\frac{2s-2}{3s-2}} M^{\frac{s}{3s-2}}.
\end{equation*}

From these, we get
\begin{align*}
    \sum_{i=1}^{M} |I(P \cap \Omega_i,\mathcal{F}_i)|
    &\leq \sum_{i=1}^{M} \left[ C_1 \left( \frac{m}{M} \right) ^{\frac{2s-1}{3s-2}+\epsilon} |\mathcal{F}_i|^{\frac{2s-2}{3s-2}} + C_2 \left( \frac{m}{M} + |\mathcal{F}_i| \right) \right]\\
    &\leq C_1 \left( \frac{m}{M} \right) ^{\frac{2s-1}{3s-2}+\epsilon} \sum_{i=1}^{M} |\mathcal{F}_i|^{\frac{2s-2}{3s-2}} + C_2 m + C_2 \sum_{i=1}^{M} |\mathcal{F}_i|\\
    &\leq C_1 12^{\frac{s-1}{3s-2}} M^{-\epsilon} m^{\frac{2s-1}{3s-2}+\epsilon} n^{\frac{2s-2}{3s-2}} + C_2 \left( m + D^2n \right),
\end{align*} where we used that $D \leq (12M)^{1/4}$.

If we choose $M$ large enough, say $M \geq (4 \times 12^{1/3})^\frac{1}{\epsilon}$, we get
\begin{equation*}
    \sum_{i=1}^{M} |I(P \cap \Omega_i,\mathcal{F}_i)| \leq \frac{1}{2} C_1 m^{\frac{2s-1}{3s-2}+\epsilon} n^{\frac{2s-2}{3s-2}} + C_2 \left( m + D^2n \right).
\end{equation*}
Notice that
\begin{equation*}
   m^{\frac{1}{2}}\leq n \leq m^{2-\frac{1}{s}} \implies m^{\frac{2s-1}{3s-2}} n^{\frac{2s-2}{3s-2}} \geq \max\{m,n\}.
\end{equation*}
Thus, we can guarantee
\begin{equation*}
\frac{1}{3}C_1 m^{\frac{2s-1}{3s-2}+\epsilon} n^{\frac{2s-2}{3s-2}} \geq C_2 (m+D^2n),
\end{equation*} by taking $C_1$ to be a large constant depending on $M$ and $C_2$. Thus, the second term can be absorbed into the first term. Therefore,
\begin{align}
    \notag \sum_{i=1}^{M} |I(P \cap \Omega_i,\mathcal{F}_i)|
    &\leq \frac{1}{2} C_1 m^{\frac{2s-1}{3s-2}+\epsilon} n^{\frac{2s-2}{3s-2}} + \frac{1}{3}C_1 m^{\frac{2s-1}{3s-2}+\epsilon} n^{\frac{2s-2}{3s-2}}\\
    &\leq \frac{5}{6} C_1 m^{\frac{2s-1}{3s-2}+\epsilon} n^{\frac{2s-2}{3s-2}}. \label{eq:offvariety}
\end{align}

Next, we estimate the incidences on the surface $\{ Q=0 \}$. In order to avoid singular points on $\{ Q=0 \}$, we perform a technical decomposition by introducing the following sequence of hypersurfaces $S_0, S_1, \dots, S_D$ and setting $S_0 := \{ Q=0 \}$, $Q_0 := Q$, and
\begin{equation*}
    Q_{i+1} = \sum_{j=1}^{4} \alpha_j^{(i)} \frac{\partial}{\partial x_j}Q_i \quad \text{for $1 \leq i \leq D$,}
\end{equation*} which defines the surface $S_{i+1} = \{ Q_{i+1} = 0 \}$, where the $\alpha_j^{(i)}$ are generic reals. All $S_i$ have degree at most $D$, and for each point $p \in P \cap \{ Q=0 \}$, there is an index $i = \text{ind}(p)$ between $0$ and $D$ which is the smallest $i$ such that $p \notin S_{i+1}$. We remark that this makes $p$ a smooth point of $S_i$, so that the tangent space to $S_i$ at $p$ is three-dimensional. Thus, the tangent space to $S_i$ at $p$ may contain at most one 2-flat in $\mathcal{F}$, and at most one 2-flat in $\mathcal{F}$ incident to $p$ is entirely contained in $S_i$. All other 2-flats incident to $p$ intersect $S_i$ in a one-dimensional curve. We consider the set of points $P_i = \{ p \in P:\ \text{ind}(p) = i \}$ and bound the number of incidences between $P_i$ and $\mathcal{F}$. The 2-flats on $S_i$ give at most $|P_i|$ incidences, so it is sufficient to consider the incidences between $P_i$ and the set of 2-flats intersecting $S_i$ in a curve. We call these intersection curves $\mathcal{V}_i$ for every $0 \leq i \leq D$. The curves in $\mathcal{V}_i$ are one-dimensional algebraic curves of some constant degree (depending on $D$). Let us project the points $P_i$ and the curves $\mathcal{V}_i$ onto a generic plane, and apply Theorem~\ref{threshold2incidence} with $d=2$ (because of projection) and $\sigma = 2- 1/s$ (because of Theorem~\ref{main}). Thus, for each $0 \leq i \leq D$ we obtain
\begin{equation*}
    |I(P_i, \mathcal{V}_i)|\leq C_3\left( |P_i|^{\frac{2s-1}{3s-2}+\epsilon}|\mathcal{V}_i|^{\frac{2s-2}{3s-2}} + |P_i| + |\mathcal{V}_i| \right),
\end{equation*} where $C_3$ is a constant depending on $D$, $\epsilon$, $s$, and $t$. Thus, we have that the number of incidences on $\{ Q = 0 \}$ is at most
\begin{equation}\label{eq:onvariety}
    \sum_{i=0}^{D} |I(P_i, \mathcal{V}_i)| + \sum_{i=0}^{D} |P_i| \leq C_3(D+1)\left(m^{\frac{2s-1}{3s-2}+\epsilon'} n^\frac{2s-2}{3s-2} + m + n\right) + m
\end{equation}
Thus, if we choose constants $C_1$, $C_2$ to be sufficiently large depending on $C_3$ and $D$, we can upper bound the quantity in \eqref{eq:onvariety} by one-sixth of the first term in \eqref{eq:grid_2flat}. Therefore, we can combine \eqref{eq:offvariety} and \eqref{eq:onvariety} to conclude the inductive proof.
\end{proof}

\begin{proof}[Proof of Theorem~\ref{grid}]
Observe that, after a point-line duality (see Section~5.1 of \cite{matousek2013lectures}), an $s$-by-$s$ grid corresponds to a copy of $K_{s,s}'$ in the incidence graph. Identify $\mathbb{C}$ with $\mathbb{R}^2$ and apply Theorem~\ref{grid_2flat} to $P$, $L$, and $\mathbb{C}^2 = \mathbb{R}^4$ to obtain
\begin{align*}
    |I(P,L)| &\leq C_1 n^{\frac{2s-1}{3s-2}+\epsilon}n^{\frac{2s-2}{3s-2}}+C_2(n+n)\\
    &\leq C_1 n^{\frac{4s-3}{3s-2}+\epsilon} + 2C_2n\\
    &\leq C_4 n^{\frac{4}{3}-\frac{1}{9s-6}+\epsilon},
\end{align*} where $C_4$ is some sufficiently large constant depending on $C_1$ and $C_2$.
\end{proof}

\section{Distinct distances under local conditions}\label{sec_distance}

In this section, we prove Theorem~\ref{distance} by transforming it into a point-variety incidence problem in $\mathbb{R}^4$, thereby allowing us to apply Theorems~\ref{main} and~\ref{threshold2incidence}. We remark that each incidence in our framework will correspond to a distance repetition, i.e. two pairs of points with the same distance, but the relationship between the number of distinct distances and the number of distance repetitions is subtle. For example, if we are given distance repetitions $|ab| = |cd|$, $|cd| = |xy|$, and $|ab| = |xy|$, then the third repetition gives no useful information. To overcome this subtlety, we leverage the fact that our $K_{s,t}'$-free incidence bound does not depend on $t$ asymptotically (together with some bookkeeping).

\begin{proof}[Proof of Theorem~\ref{distance}]

Let $P$ be the set of $n$ points such that every $p$ points determine at least $q$ distinct distances, where
\begin{equation*}
        q = \binom{p}{2} - p + 3 \cdot \left\lfloor \frac{p}{2s} \right\rfloor + 2s + 2.
\end{equation*}

We define \begin{equation*}
        \mathcal{E} = \{ (a,b,c,d) \in P^4: |ac| = |bd| > 0 \}
\end{equation*} where $|ab|$ denotes the distance between $a,b \in \mathbb{R}^2$. Suppose the number of distinct distances determined by $P$ is $r$ and suppose these distances are $d_1, \ldots, d_r$. We write $E_i = \{ (a,b) \in P^2 : |ab| = d_i\}$ so that $\sum_{i=1}^{r} |E_i| = n(n-1)$. Using Jensen's inequality, we have
\begin{equation*}
    |\mathcal{E}| = \sum_{i=1}^{r} |E_i|^2 \geq \frac{1}{r} \left(\sum_{i=1}^r |E_i| \right)^2  = \Omega\left(\frac{n^4}{r} \right).
\end{equation*} Consequently, it suffices for us to prove that
\begin{equation}\label{energy_distance}
    |\mathcal{E}| \leq O\left(n^{\frac{20}{7} - \frac{18}{7(7s-4)} + \epsilon} \right).
\end{equation}

To this end, for every pair $(a,b) \in P^2$, we define a corresponding point in $\mathbb{R}^4$ by $v_{a,b} = (a_x, a_y, b_x, b_y)$. Here $a_x$ (respectively $a_y$) is the $x$-coordinate (respectively $y$-coordinate) of the plane point $a$. We also define a quadric in $\mathbb{R}^4$ as $Q_{a,b} : (x -a_x)^2 + (y - a_y)^2 = (z - b_x)^2 + (w - b_y)^2$. We observe that if $v_{a,b}$ lies on $Q_{c,d}$ then it follows $|ac| = |bd|$. We consider a random equitable partition of the point pairs $P^2 = P_1 \sqcup P_2$, that is, $\left||P_1|-|P_2|\right|\leq 1$ and we pick such a partition uniformly at random. Note that the probability for two pairs $(a,b),(c,d)\in P^2$ to be in the same part in this partition is at least $1/2$. Define $\mathcal{P} = \{ v_{a,b} : (a,b) \in P_1\}$ and $\mathcal{Q} = \{ Q_{a,b}: (a,b) \in P_2\}$. Then our observation, together with a basic probabilistic argument, tell us that
\begin{equation*}
    |\mathcal{E}| \leq  2 \cdot |I(\mathcal{P}, \mathcal{Q})| + O(n^2).
\end{equation*} It suffices for us to show the incidence graph $I(\mathcal{P}, \mathcal{Q})$ does not contain a copy of $K_{s, t}'$ with $t = (2s + p)^2 + 1$. Indeed, if this is given, we apply our Theorems~\ref{main} and~\ref{threshold2incidence} with $d = 4$ and $\sigma = 2 - 1/s$ to conclude~\eqref{energy_distance}. For the rest of our proof, the word ``point'' refers to a point in $\mathbb{R}^2$, and we say a point $x$ is in a vertex $v_{a,b}$ (or $Q_{a,b}$) if $p \in \{a,b\}$.

For the sake of contradiction, we suppose there is a copy of $K_{s,t}'$ in $I(\mathcal{P}, \mathcal{Q})$. We shall describe an iterative procedure of processing the edges of this $K_{s,t}'$ which will produce a set $\mathcal{A}$ of $p$ points determining less than $q$ distinct distance, contradicting our hypothesis on $P$. We denote the vertex set of $K_{s,t}'$ as $S \sqcup T\sqcup K$ where $|S| = s$, $|T|=t$, and, $|K|=st$ under the natural identification. We let $\mathcal{S}$ be the set of points in vertices of $S$. Note that we have $|\mathcal{S}| \leq 2s$. We also identify a subset $T' \subset T$ with the following property: $|T'| = t' > p$ and there is an enumeration $v_{a_1, b_1}, \dots, v_{a_{t'}, b_{t'}}$ of $T'$ such that\begin{equation*}
    \{ a_i, b_i \} \not \subset \{ a_1, \ldots, a_{i-1}, b_1, \ldots, b_{i-1}\} \cup \mathcal{S}.
\end{equation*} The existence of such $T'$ is guaranteed by $t$ being large enough.

Now we describe our edge-processing procedure. Let us write $S = \{u_1,\dots, u_s\}$, $T'=\{v_1,\dots, v_{t'}\}$, and $w_{i,j}$ the vertex in $K$ adjacent to both $v_i$ and $u_j$. Initially, we set $\mathcal{A} = \mathcal{S}$. During our procedure, we will add points into $\mathcal{A}$, and we will label some point pairs in $\mathcal{A}$ by other pairs. When a pair $ac$ is labeled by another pair $bd$, we will always have $|ac| = |bd|$, and our intuition is that $ac$ will not ``contribute'' to the number of distinct distance determined by $\mathcal{A}$. Our procedure has at most $t'$ rounds but may terminate early. During the $i$-th round, we consider the vertices $w_{i,1},\dots, w_{i,s}$ in order and process the two edges $u_jw_{i,j}$ and $v_iw_{i,j}$ in order. When an edge with end-vertices $v_{a,b}$ and $Q_{c,d}$ is processed, we add the points $a,b,c,d$ into $\mathcal{A}$ and try to produce a labeled pair in the following way: if both pairs $ac$ and $bd$ are unlabeled, we label $ac$ with $bd$; otherwise, if $ac$ is unlabeled and $bd$ is labeled with a pair $xy$ different from $ac$, we label $ac$ with $xy$; otherwise, if $bd$ is unlabeled and $ac$ is labeled with a pair $xy$ different from $bd$, we label $bd$ with $xy$. Let us remark that because $P_1$ and $P_2$ are disjoint, $ac$ and $bd$ will be pairs, i.e. $a\neq c$ and $b\neq d$. If at a certain step, adding the points into $\mathcal{A}$ would result in $|\mathcal{A}| > p$, we terminate the whole procedure before adding the points. We can notice that $|\mathcal{A}|$ increases by at most two from each edge, because there is always one end-vertex that appeared before. The properties of $T'$ guarantee that our procedure will end with $|\mathcal{A}| = p$ or $|\mathcal{A}| = p-1$. In the latter case, we add an arbitrary point from $P$ into $\mathcal{A}$ so that we end up with a set of $p$ points.

During the $i$-th round of our procedure, we use $p_i$ to denote the number of points newly added into $\mathcal{A}$ and $\ell_i$ to denote the number of labeled pairs we produced. By analyzing our procedure through some careful bookkeeping, we can verify the following claim.
\begin{claim}\label{claim_distance}
    If $\ell_i \leq 2s - 2$, then $p_i \leq \ell_i$.
\end{claim} Given this claim, we argue the $p$ points in $\mathcal{A}$ determine less than $q$ distinct distances. We write \begin{equation*}
    x  = |\{i : p_i = 2s\}|,\quad y  = |\{i : p_i = 2s+1\}|,\quad z  = |\{i : p_i = 2s+2\}|.
\end{equation*} Note that $p_i \leq 2s+2$ always holds. Since $2s \cdot x + (2s + 1) \cdot y + (2s + 2) \cdot z \leq p$, we have\begin{equation*}
    x + y + z \leq \left\lfloor \frac{p}{2s} \right\rfloor.
\end{equation*} From our Claim~\ref{claim_distance}, we can argue
\begin{equation*}
    \sum_{i}c_i + x + 2y + 3z \geq \sum_i p_i \geq p - 1 - |\mathcal{S}| \geq p - 2s -1.
\end{equation*} Combining these two identities, we conclude
\begin{equation*}
    \sum_i c_i \geq p - 2s -1 - 3(x + y + z) \geq p - 2s - 1 - 3 \cdot \left\lfloor \frac{p}{2s} \right\rfloor.
\end{equation*} Observe that the distance of each labeled pair equals the distance of its label, which is another unlabeled point pair. Therefore, the number of distinct distances determined by the $p$ points in $\mathcal{A}$ is at most \begin{equation*}
    \binom{p}{2} - \sum_i c_i \leq q - 1,
\end{equation*} which contradicts our hypothesis on $P$.

\medskip

We devote the rest of our proof to Claim~\ref{claim_distance}. Suppose we have $\ell_i \leq 2s - 2$ for some index $i$, then there are $2s - \ell_i$ many edges that did not produce any labeled pairs. We call them \textit{bad edges} and their end-vertices in $K$ \textit{bad vertices}. Note the following important property of our procedure: no new points are added into $\mathcal{A}$ while processing a bad edge. Indeed, one possibility for an edge $v_{a,b}Q_{c,d}$ to be bad is that $bd$ is already labeled with $ac$, then these points must have appeared before and hence are already in $\mathcal{A}$. The other possibilities can be argued similarly. Therefore, the points added into $\mathcal{A}$ during the $i$-th round must be in either $v_i$ or a \textit{good} (i.e. not bad) $w_{i,j}$ ($1\leq j\leq s$). We also call $v_i$ a good vertex as a convention.

We call a vertex $w_{i,j}$ \textit{singular} if it is adjacent to exactly one bad edge. We call a point $a$ in a good vertex $v_{a,b}$ \textit{redundant} if $a = b$, or it is in another good vertex, or appeared in previous rounds, or in $\mathcal{S}$. The definition for redundant point is similar for $b$ in $Q_{a,b}$. We proceed with our proof under the assumption that $2s-\ell_i$ is odd: we argue that there are at most one singular vertex and no redundant points, otherwise our Claim~\ref{claim_distance} is true. Indeed, if there are two singular vertices, the total number of bad vertices is at least \begin{equation*}
    \left\lceil \frac{2s - \ell_i - 2}{2} \right\rceil + 2 = \frac{2s-\ell_i + 3}{2}.
\end{equation*} Hence we can compute $p_i < \ell_i$ using $p_i$ being at most twice the number of good vertices. If there is a redundant point, then $p_i$ is at most twice the number of good vertices minus one (since we do not need to count the redundant point), which is at most\begin{equation*}
    2 \cdot \left(s + 1 - \left\lceil \frac{2s - \ell_i}{2} \right\rceil \right) - 1 = \ell_i.
\end{equation*}

If no bad edge is adjacent to $v_i = v_{a_1,b_1}$, there are $2s-\ell_i \geq 2$ singular vertices and we are done. Thus, suppose $v_{a_1,b_1} Q_{a_2,b_2}$ is the first bad edge adjacent to $v_i$ we encounter. Since we do not produce a labeled pair while processing this edge, we must know from previous edges that \begin{equation} \label{eq:distance}
    |a_1a_2| = |c_1c_2| = |c_3c_4| = \dots = |d_1d_2| = |b_1b_2|.
\end{equation} We can assume that all equal signs here are produced by distinct good edges, otherwise this equation can be simplified. We also assume $a_1\neq b_1$ otherwise there is a redundant point in $v_i$.

Now we consider the edge $e_1$ that produced $|a_1a_2| = |c_1c_2|$. If $e_1$ is processed in previous rounds, then $a_1$ in $v_i$ is redundant and we are done. Hence we assume $e_1$ is processed in this round, and our proof diverges into the following cases:\begin{itemize}

    \item If $e_1$ is not adjacent to $v_{a_1,b_1}$ or $Q_{a_2,b_2}$, we consider the end-vertex $w \in K$ of $e_1$. Notice $w$ is already adjacent to a good edge $e_1$, and $v_i w$ cannot be bad otherwise $v_{a_1,b_1} Q_{a_2,b_2}$ is not the first bad edge adjacent to $v_i$ we encounter. Thus $w$ is a good vertex, and the point $a_1$ is either in $w$ or in $\mathcal{S}$, making $a_1$ in $v_i$ redundant.
    
    \item If $e_1$ is adjacent to $v_{a_1,b_1}$, upon renaming $c_1$ and $c_2$, there are two possibilities. (To see they are all the cases, one may consider where $a_1$ in $|a_1a_2| = |c_1c_2|$ comes from.)
    \begin{itemize}
        \item $e_1 = v_{a_1,b_1} Q_{a_2,c_2}$ with $b_1=c_1$. In this case, the pairs $b_1b_2$ and $c_1c_2$ must be different, otherwise we must have $b_2 = c_2$, making $v_{a_1,b_1}Q_{a_2,b_2}$ simultaneously good and bad. In particular, the term $|c_3c_4|$ exists in \eqref{eq:distance}. Denote the edge that produced $|c_1c_2|=|c_3c_4|$ as $e_2$ and its end-vertex in $K$ as $w'$. If $e_2$ is processed in previous rounds or $c_1$ is not in $w'$, then $b_1=c_1$ in $v_i$ will be redundant and we are done. Hence we assume $e_2$ is a good edge during the $i$-th round and the point $c_1$ is in $w'$. Since $w'$ and $Q_{a_2,c_2}$ cannot both be singular, one of them must be a good vertex. If $w'$ is a good vertex, $b_1 = c_1$ in $v_i$ will be redundant. If $Q_{a_2,c_2}$ is good, $c_2$ in $Q_{a_2,c_2}$ is redundant, because the point $c_2$ appears in the left end-vertex of $e_2$, hence is in $\mathcal{S}$.
        
        \item $e_1 = v_{a_1,b_1} Q_{c_2,a_1}$ with $a_1=c_1$ and $b_1=a_2$. In this case, the pairs $b_1b_2$ and $c_1c_2$ must be different, otherwise we must have $c_1=b_2=a_2$ and $c_2=b_1=a_2$, making the edge $v_{a_1,b_1}Q_{a_2,b_2}$ simultaneously good and bad. Similarly, $|c_3c_4|$ exists and we define $e_2$ and $w'$. If $e_2$ is processed in previous rounds or $c_1$ is not in $w'$, $a_1 = c_1$ in $v_i$ will be redundant. If $e_2$ is a good edge during the $i$-th round and the point $c_1$ is in $w'$, $a_1 = c_1$ in $v_i$ is still redundant because it appears in both $w'$ and $Q_{c_2,a_1}$, among which at least one is good (otherwise there will be two singular vertices).
    \end{itemize}

    \item If $e_1$ is adjacent to $Q_{a_2,b_2}$, the point $a_1$ must be in $Q_{a_2,b_2}$, otherwise it will be in the end-vertex of $e_1$ inside $S$, making $a_1$ in $v_i$ redundant. Also, $a_1 \neq a_2$ because $a_1a_2$ is a pair. Thus, upon renaming $c_1$ and $c_2$, we have $e_1 = v_{c_1,a_2}Q_{a_2,b_2}$ with $a_1 = b_2$ and $a_2=c_2$. Notice that the left end-vertex of $e_1$ is also $v_{c_1,c_2}$, hence $b_1 \in \{c_1,c_2\}$ would imply $b_1$ in $v_i$ being redundant. Thus, we assume $b_1 \not\in \{c_1,c_2\}$ and consequently the pairs $b_1b_2$ and $c_1c_2$ are different. In particular, the term $|d_1d_2|$ exists in \eqref{eq:distance}. In this case, we also know that $Q_{a_2,b_2}$ is the only singular vertex. Then the edge that produced $|d_1d_2| = |b_1b_2|$ will make $b_1$ in $v_i$ redundant.
\end{itemize}

    This finishes the proof under the assumption that $2s-\ell_i$ is odd. The proof with $2s-\ell_i$ being even follows from very similar arguments. In that case, we allow no singular vertex and wish to find two redundant points. We skip the proof for $2s-\ell_i$ being even and conclude Claim~\ref{claim_distance}, hence also the proof of Theorem~\ref{distance}.
\end{proof}

\section{Remarks}\label{sec_remark}

We mention three related open problems (eventually only two). The first problem asks whether we can say $\sigma_s = 2 -1/s$ is indeed the linear threshold of $K_{s,t}'$ for $s \geq 1$ and large $t$. Essentially, the problem concerns the dependence of $t$ on $\sigma$ in Theorem~\ref{main}.
\begin{problem}
    Is the following true: for every integer $s \geq 1$, any $\sigma < \sigma_s$ is less than the linear threshold of $K_{s,t}'$ for sufficiently large $t$ depending only on $s$?
\end{problem}

The second problem asks for an extension of Theorem~\ref{main} to hypergraph subdivisions. We use $K_s^r$ to denote the complete $r$-partite $r$-uniform hypergraph with each part of $K_s^r$ having $s$ vertices, and we use $(K_s^{r})'$ to denote its subdivision: $(K_s^{r})'$ has left hand side vertex set $\{(i,j): 1\leq i\leq r,~ 1\leq j\leq s\}$, right hand side vertex set $\{(j_1,j_2,\dots, j_r): 1\leq j_1,j_2,\dots,j_r \leq s\}$, and two vertices $(i,j)$ and $(j_1,j_2,\dots, j_r)$ are adjacent if and only if $j_i = j$.
\begin{problem}\label{hypergraph}
    Is the linear threshold of $(K_s^{r})'$ strictly less than $r$?
\end{problem}

\begin{remark}
    After an earlier version of this paper was posted on arXiv, Oliver Janzer commented that Problem~\ref{hypergraph} can be answered in the affirmative following the ideas in \cite{janzer2019improved}: We fix an arbitrary bipartite graph $G = (U \sqcup V, E)$ with $|V| \geq |U|^{r - \epsilon}$ and $|E| \geq \delta |V|$, and our goal is to show $G$ contains a copy of $(K_s^{r})'$ when $\epsilon$ and $\delta$ are properly chosen; Firstly, we can assume $G$ is close to being biregular using an argument similar to our Lemma~\ref{biregular}; Secondly, we call a $r$-set $\{u_1,\dots,u_r\} \subset U$ \textit{heavy} if $u_1,\dots,u_r$ have at least $s^r$ common neighbors, and we call it \textit{light} if there are at least one but less than $s^r$ common neighbors; We can assume there is no $K_s^r$ formed by heavy $r$-sets in $U$ (otherwise we can find a copy of $(K_s^{r})'$ in $G$ greedily), then using an argument similar to Lemma~10 in \cite{janzer2019improved}, we can argue that there are $\Omega(|U|^{r - \epsilon})$ light $r$-sets in $U$; Now, the Erd\H{o}s--Simonovits supersaturation theorem~\cite{erdHos1983supersaturated} guarantees the existence of $O(n^{rs - s^r\epsilon})$ many copies of $K_s^r$ formed by light $r$-sets in $U$; Finally, every $K_s^r$-copy in $U$ should correspond to a $(K_s^{r})'$-copy in $G$ except for the degenerate case where there are less vertices on the right hand side, but the number of such exceptions can be upper bounded by $o(n^{rs - s^r\epsilon})$, completing the proof. Here, the assumption of $G$ being close to regular is used in the last step, and $\epsilon,\delta$ are chosen to make inequalities hold throughout the proof.
\end{remark}

\begin{remark}
    Suppose $P$ is a set of $m$ points and $\mathcal{H}$ is a set of $n$ hyperplanes, both in $\mathbb{R}^d$, such that the incidence graph $I(P,\mathcal{H})$ does not contain a copy of $K_{s,t}$, Apfelbaum and Sharir~\cite{apfelbaum2007large} proved that\begin{equation}\label{hyperplane}
        |I(P,\mathcal{H})| \leq O\left( (mn)^{\frac{d}{d+1}} + m + n \right),
    \end{equation} where the constant hidden in the $O$-notation depends on $d,s,t$. Combining the affirmative answer to Problem~\ref{hypergraph} and our Theorem~\ref{threshold2incidence}, we can obtain a polynomial improvement of \eqref{hyperplane} with the additional hypothesis that $I(P,\mathcal{H})$ does not contain a copy of $(K_s^{d})'$. We remark that a $(K_s^{d})'$-copy in $I(P,\mathcal{H})$ naturally corresponds to a $d$-dimensional grid structure. However, it is unknown whether \eqref{hyperplane} is asymptotically tight beyond the case $d=2$ (see~\cite{sudakov2024evasive}), so \eqref{hyperplane} could possibly be improved without forbidding any substructures.
\end{remark}

The third problem is from incidence geometry. Our Theorem~\ref{threshold2incidence} fails to answer this problem due to the linear threshold of $C_6$ being 2 as shown by de Caen and Sz\'ekely~\cite{de1991maximum}. We call a point-line configuration $(P, L)$ a \textit{triangle} if $|L|=3=|P|$ and $P = \{\ell_1 \cap \ell_2 : \ell_1, \ell_2 \in L\}$.
\begin{problem}
    Is the following true: any $n$ points and $n$ lines in the real plane without a triangle sub-configuration determine $O(n^{4/3 - c})$ incidences for some constant $c > 0$?
\end{problem}

\noindent {\bf Acknowledgement.} This work was initiated during the Convex and Discrete Geometry Workshop at Erd\H{o}s Center in September 2023. Anqi Li would like to thank Adam Sheffer for introducing her to the problem of $d(n,p,q)$ during the 2020 NYC Discrete Math REU and Yufei Zhao for supporting her visit to the Erd\H{o}s Center. We wish to thank David Conlon, Oliver Janzer, Xizhi Liu, Dániel Simon, and Andrew Suk for helpful discussions.

\bibliographystyle{abbrv}
{\footnotesize\bibliography{main}}

\end{document}